%% file: lqr_high_prob.tex
\documentclass[9pt]{article} 
\usepackage{hyperref}
\usepackage{url}
\usepackage{smile}
\usepackage{graphicx} 
\usepackage{algorithm}
\usepackage{algorithmic}
\usepackage{epstopdf}
\usepackage[margin=1.0in]{geometry}
\usepackage[export]{adjustbox}
\usepackage{mathtools, cuted}
\usepackage{bbm}
\usepackage{wrapfig}
\usepackage{subcaption}
\usepackage{caption}
\usepackage{enumitem}
\usepackage{tabularx}
\usepackage{smile}
\usepackage{tcolorbox}
\usepackage{enumitem}
\usepackage{mathtools, cuted}
\usepackage{caption}
\usepackage{subcaption}

\numberwithin{equation}{section}

\usepackage[nocompress]{cite}

\hypersetup{
    colorlinks=true,
    linkcolor=blue,
    anchorcolor=blue, 
    citecolor=blue,
}

\allowdisplaybreaks[4]

\renewcommand{\frac}[2]{\tfrac{#1}{#2}}

\title{(Cheap) Stochastic Policy Gradient Converges with High Probability for Linear Quadratic Regulator}
\author{
Yan Li\thanks{Department of Industrial and Systems Engineering,
Texas A\&M University. (E-mail: \url{yan.li@tamu.edu}).}
\and 
Chengze Xie\thanks{
Department of Industrial and Systems Engineering,
Texas A\&M University. (E-mail: \url{changze@tamu.edu}). 
}
}

\date{\vspace{-6ex}}
\begin{document}
\maketitle

\begin{abstract}
We study the convergence of the vanilla stochastic policy gradient method applied to the linear quadratic regulator (LQR) problem. The method is cheap in the following sense: (1) at each iteration only $\tilde{\cO}(1)$ interactions with the environment are needed, therefore allowing frequent policy improvement steps, and (2) to ensure stability throughout and convergence to an $\epsilon$-optimal policy with probability $1-\delta$, only $\cO({\tt Polylog}(1/\delta)/\epsilon)$ interactions are needed. To the best of our knowledge, this appears to be the first time that a stochastic model-free policy optimization method for LQR converges with high probability with $\tilde{\cO}(1)$ per-iteration computation and polylogarithmic dependence on the confidence level. The convergence analysis presented here is agnostic to LQR specifics and hence could be potentially generalized to a broader class of problems.
\end{abstract}

\input{intro}

\input{spg}

\input{convergence}

\input{oracle}

\input{conclusion}


\bibliographystyle{plain}
\bibliography{reference}

\end{document}

%% file: intro.tex

\section{Introduction}

We consider the following  undiscounted infinite-horizon linear quadratic regulator (LQR) problem:\footnote{
As will be seen in Section \ref{condition_noise}, the analysis to be discussed in this manuscript can be extended to the discounted-cost and the average-cost settings without essential changes. 
}
\begin{align}\label{lqr_obj}
\min_{K \in \RR^{m \times n}}  f(K)=\EE_{x_0 \sim \cD} \sbr{ \tsum_{i\geq 0}(x_i^\top Q x_i+a_i^\top R a_i) } ,
\end{align} 
where $Q \in \RR^{n \times n}$, $R \in \RR^{m \times m}$ are positive definite matrices, and  
\begin{align}\label{eq_dynamics}
x_{i+1}=Ax_i+Ba_i,~ a_i=-Kx_i,  ~ \EE_{x_0 \sim \cD}  \sbr{x_0 x_0^\top} = I_n.
\end{align}
We assume in addition that $\norm{x_0} \leq D_X$ for some $D_X > 0$. 
To facilitate our discussion, 
for any policy $K$, we define $\Delta(K) = f(K) - f^*$, where $f^* = \min_{K} f(K)$. 
We assume there exists a policy with finite cost, so that $f^* < \infty$. 
It is well known that the restriction to linear Markov policies in \eqref{eq_dynamics} suffices to obtain optimality for \eqref{lqr_obj} over the class of all history-dependent nonlinear policies \cite{kalman1960general, kalman1960contributions}.

We can roughly categorize the solution methods for solving LQR into three classes. Let us first operate in the idealized setting where the underlying model parameters $(A, B, Q, R)$ are known. 
The first class 
builds upon the observation that the optimal policy $K^*$ of \eqref{lqr_obj} can be computed via solving the algebraic Riccati equation (ARE), and therefore designs specialized numerical linear algebra methods for solving the ARE \cite{laub1979schur,vandooren1981generalized}. 
The second class of methods based on the dynamic programming principle either operates in the policy space  \cite{hewer1971iterative,bertsekas1995dynamic}
or value space \cite{bellman1957dynamic,bertsekas1995dynamic}.
In recent years, a new class of first-order methods based on nonlinear programming has received considerable attention in MDPs  \cite{kakade2001natural,lan2023policy,agarwal2021theory,xiao2022convergence} and later in LQR \cite{fazel2018global,malik2020derivative}. Such methods iteratively update the policy via its first-order (gradient) information. 
The global convergence of (natural) policy gradient applied to LQR,  despite the latter being a non-convex function, has been established in \cite{fazel2018global}, followed by an active line of research \cite{fatkhullin2021optimizing, malik2020derivative,yang2019global,zhou2023single,zeng2024two,ju2025model,moghaddam2025lens}.

Let us now consider the setting when the exact parameters are unknown. Solution methods in this case largely operate based on partially learning about the system by interacting with the environment.\footnote{In this manuscript we refer to one interaction as one state transition after committing an action. We choose not to refer to such interactions as samples as in some prior work \cite{fazel2018global,malik2020derivative}, since there is no randomness in such transitions for the undiscounted setting.
Note that LQR with stochastic transitions only makes sense for the discounted-cost or average-cost settings. 
} 
In principle all three aforementioned classes of solution methods have their counterparts in this setting.
A natural idea underlying such development is to first  estimate the system parameters after collecting enough interactions, followed by solving a nominal LQR with the estimated parameters \cite{mania2019certainty}.
We refer to this idea as the model-based principle. 
A particular question is therefore to determine the number of interactions needed for dynamics estimation in order to find an $\epsilon$-optimal policy.
This question has been  discussed in \cite{dean2020sample,mania2019certainty} for the average-cost setting with potential stochastic transition dynamics.
It should be noted that for \eqref{lqr_obj}, a constant number of interactions is sufficient for forming a linear system that uniquely identifies system parameters provided the query state-action pair provides sufficient coverage \cite{vanwaarde2020data}. 

An alternative principle for designing interaction-based methods for LQR is to avoid explicit estimation of model parameters altogether, which we refer to as model-free methods.
This has been considered, for instance, for policy-based methods \cite{bradtke1994adaptive,jiang2012computational} and value-based methods \cite{bian2016value,jiang2024adaptive,lai2025robust}. 
In particular, it has been shown that $\tilde{\cO}(1/\epsilon)$ environment interactions are needed for policy iteration \cite{krauth2019finite} and $\tilde{\cO}(1/\epsilon)$ for value iteration \cite{lai2025robust}.
Similar progress has been made for first-order (gradient-based) methods \cite{fazel2018global,malik2020derivative,yang2019global,zhou2023single,zeng2024two,ju2025model,moghaddam2025lens}. 
For instance, $\cO({\tt Poly}(1/\epsilon))$ interaction complexity has been first established in  \cite{fazel2018global}, followed by $\tilde{\cO}(1/\epsilon^5)$  in \cite{yang2019global},  $\tilde{\cO}(1/\epsilon^2)$  in \cite{malik2020derivative}, $\tilde{\cO}(1/\epsilon^{3/2})$  in \cite{zeng2024two}, and more recently, $\tilde{\cO}(1/\epsilon)$ in \cite{malik2020derivative,zhou2023single,ju2025model,moghaddam2025lens}.

Why would one consider model-free methods over model-based ones?
We can perhaps argue that model-based methods are somewhat expensive, in the sense that a sufficiently large number of interactions  need to be collected before we start to improve the policy \cite{dean2020sample,mania2019certainty}, with non-trivial computation consumed between consecutive policy improvements \cite{hewer1971iterative,bertsekas1995dynamic}. 
In this sense a reasonably good model-free method should waste no time in collecting too many environment interactions or consume too much computation for each policy improvement, and should ensure incremental yet  steady progress towards the optimal policy. 
Unfortunately some of the aforementioned model-free methods with strong performance guarantees do not enjoy such properties. 
In essence they operate in a way that the methods  behave similarly to their deterministic counterparts, which assume exact model information.
This is typically achieved by collecting a large number of interactions (via long trajectories or mini-batches) at each round of policy improvement to ensure that the stochastic update to the policy is close to its deterministic counterpart  \cite{krauth2019finite,fazel2018global,yang2019global,malik2020derivative,ju2025model,moghaddam2025lens}. 

In this manuscript we are interested in the convergence of cheap stochastic policy optimization methods.
By cheap we mean that they use $\tilde{\cO}(1)$ interactions and computation between consecutive policy improvements. 
The small number of interactions implies that the update is truly stochastic instead of being a close approximation of deterministic updates, and accompanied with $\tilde{\cO}(1)$ computation, ensures fast policy improvement. 
We will focus on cheap gradient-based methods \cite{malik2020derivative,moghaddam2025lens,zeng2024two,zhou2023single}. 
Existing convergence certificates come either with constant probability \cite{malik2020derivative,moghaddam2025lens,fazel2018global,mohammadi2021linear}, or in expectation \cite{zeng2024two,zhou2023single}. 
It appears quite surprising that no high-probability results with polylogarithmic dependence on the confidence level have been established in the prior literature, despite fruitful development of boosting expectation to high-probability statements in stochastic optimization \cite{lan2020first,davis2021low,nemirovskij1983problem,ji2026computation}.
It should be noted that the aforementioned convergence in expectation makes the non-trivial assumption on uniform stability throughout policy optimization \cite{zeng2024two,zhou2023single}.
Indeed, the main technical difficulty associated with obtaining convergence for cheap stochastic policy optimization methods (even in expectation) pertains to maintaining stability of the policies throughout policy optimization, as nothing can be said about the stochastic gradient once the policy becomes unstable and the objective \eqref{lqr_obj} itself blows up to infinity.
In view of this, the dependence  of gradient noise on the iterate  within LQR behaves in a completely non-smooth way compared to problems recently studied in optimization literature with unbounded domains,  which depends smoothly on the distance to the optimal solution \cite{kotsalis2022simple,ilandarideva2025accelerated,ji2026computation, telgarsky2022stochastic}.\footnote{
Of course \eqref{lqr_obj} also has an unbounded domain. This on the other hand is not our main concern within the analysis. 
}

\vspace{0.05in}

Our essential contribution in this manuscript is to provide a framework for convergence analysis that allows us to obtain such high-probability convergence guarantees. 
Our contributions can be roughly summarized as follows. 
First, we propose a rather generic framework that allows us to establish high-probability convergence of LQR with cheap stochastic policy gradient methods, provided the gradient construction oracle satisfies some minimalist assumption. 
In particular, the convergence framework uses no LQR specifics, and applies to generic black-box optimization for objectives satisfying the Polyak-Lojasiewicz (PL) property \cite{polyak1963gradient} and local smoothness, and hence can be potentially applied in broader setups. 
Second, by instantiating the framework with a concrete choice of stepsizes, we establish an $\cO(1/\sqrt{T})$ convergence rate of SPG, which is further improved to $\cO(1/T)$ by utilizing the PL property. 
Finally, we verify that the minimal stochastic gradient oracle assumption can be indeed satisfied by slight modification of an existing gradient estimator \cite{malik2020derivative}, and correspondingly establish its $\tilde{\cO}({\tt Polylog (1/\delta)}/\epsilon)$ interaction complexity with probability at least $1-\delta$.

%% file: spg.tex


\section{Stochastic Policy Gradient}\label{sec_spg}

Going forward, we will focus on the vanilla stochastic policy gradient (SPG) method applied to the LQR problem. 
At any given iteration $t$, SPG first assumes access to a stochastic oracle that produces an  estimate $G_t$ of $\nabla f(K_t)$, and then updates the policy via 
\begin{align}\label{eq_spg_update}
 K_{t+1}=K_t - \eta_t G_t, ~ \forall t \geq 0 .
\end{align}
Throughout the rest of our discussion, we assume that the initial policy $K_0$ is stable.

\begin{definition}
A controller $K$ is stable if $\rho(A-BK)<1$, where $\rho(\cdot)$ denotes the spectral radius.
\end{definition}

It is clear that with \eqref{eq_dynamics}, a policy $K$ being stable is equivalent to $\Delta(K) < \infty$. 
Let us use $\cF_t$ to denote the corresponding filtration up to iteration $t$ (excluding $t$).
Our subsequent analysis will make use of the following condition on the stochastic estimate.

\begin{condition}\label{condition_noise}
For any $\Delta > 0$, and $\delta \in (0,1)$, there exist $b \geq 0$, $V_\Delta > 0$ and $M_{\Delta,\delta} >0$, such that 
\begin{align}
 \norm{\EE[G_t| \cF_t]-\nabla f(K_t)} & \leq b, \label{condition_bound_bias} \\
 \EE[\norm{G_t}^2| \cF_t] &\leq V_\Delta,  \label{condition_bound_moment}\\
 \PP\rbr{\norm{G_t} \leq M_{\Delta,\delta}| \cF_t} & \geq 1- \delta , \label{condition_bound_grad}
\end{align}
for any $K_t$ with $\Delta(K_t) \leq \Delta$. 
\end{condition}

Clearly, Condition \ref{condition_noise} posits that the estimator $G_t$ has small noise and is bounded both in expectation and with high probability. 
It is perhaps worth noting that we allow the upper bounds of $G_t$ (either the moment bound $V_\Delta$ or the high-probability bound $M_{\Delta, \delta}$) to depend on the optimality gap. 
Such a dependence is indeed essential, as the size of the true gradient itself $\nabla f(K_t)$ would depend on the current optimality gap, and blows up to infinity if $K_t$ becomes unstable. 
For the rest of our discussion in this section, we will discuss the convergence of SPG \eqref{eq_spg_update} assuming the stochastic oracle satisfies Condition \ref{condition_noise}.
We will then discuss the construction of such an oracle in Section \ref{sec_oracle}. In particular, the constructed estimate $G_t$ will satisfy \eqref{condition_bound_grad} with probability $1$. 
 Condition \ref{condition_noise}  also appears to be quite flexible, in the sense that it does not necessarily require the estimate $G_t$ to follow light-tail distributions.

We next present two basic properties of the LQR objective \eqref{lqr_obj}, namely it is locally smooth, and satisfies a global PL property \cite{fazel2018global,malik2020derivative}.
It should be noted that these will be the only LQR-specific properties we utilize within our discussion in Section \ref{sec_convergence}. 

\begin{lemma}\label{lemma_lqr_landscape}
There exists $\mu>0$ such that for any $\Delta>0$, one can choose $L_{\Delta},r_{\Delta}>0$ so that for every policy $K$ with $\Delta(K) < \Delta$, we have that 
$K'$ is stable for any $K'$ satisfying $\norm{K' - K} \leq r_{\Delta}$, and 
\begin{align}
 f(K') & \leq f(K)+\inner{\nabla f(K)}{K' - K}
                  +\frac{L_{\Delta}}{2}\norm{K - K'}^2,     \label{ineq_local_smoothness}\\
 \mu\Delta(K) & \leq \norm{\nabla f(K)}^2
                    \leq 2L_{\Delta}\Delta(K).    \label{ineq_gradient_gap_bounds}
\end{align}
\end{lemma}

\begin{proof}
From  \cite[Lemmas 5 and 6, and discussion on page 14]{malik2020derivative},  we obtain \eqref{ineq_local_smoothness},  the stability of $K'$, and the first inequality in \eqref{ineq_gradient_gap_bounds}.
Note that without loss of generality we can assume $L_{\Delta} \geq \frac{2\Delta}{r_{\Delta}^2}$.
By taking $K' = K - r_{\Delta} \frac{\nabla f(K)}{\norm{\nabla f(K)}}$, we have 
from \eqref{ineq_local_smoothness} that
\begin{align*}
f^* \leq f(K') \leq f(K) - r_{\Delta} \norm{\nabla f(K)} + \frac{L_{\Delta}r_{\Delta}^2}{2},
\end{align*}
which implies $ \norm{\nabla f(K)} \leq \frac{1}{r_{\Delta}} \rbr{\Delta + \frac{L_{\Delta} r_{\Delta}^2}{2} } \leq L_{\Delta}r_{\Delta}$. 
Hence one can apply \eqref{ineq_local_smoothness} again with $K' = K - \nabla f(K) / L_{\Delta}$, which yields the second inequality of \eqref{ineq_gradient_gap_bounds}.
\end{proof}

It should be noted that Lemma \ref{lemma_lqr_landscape} also holds for the discounted-cost \cite{malik2020derivative} and average-cost settings \cite{yang2019global, ju2025model} with potential stochastic state transitions, and hence our subsequent analysis in Section \ref{sec_convergence} also applies to these two settings. 
In view of Lemma \ref{lemma_lqr_landscape}, for any policy $K$ with $\Delta(K) \leq \Delta$, we must have 
\begin{align}\label{bd_gd_via_opt_gap}
 \norm{\nabla f(K)}\leq G_\Delta \coloneqq2\sqrt{L_{\Delta}\Delta}.
\end{align}
Condition \ref{condition_noise}, together with Lemma \ref{lemma_lqr_landscape}, allows us to treat the LQR problem as a general black-box optimization problem.
In fact, our subsequent analysis in Section \ref{sec_convergence} requires nothing about the specifics of the LQR problem.
Hence it is perhaps reasonable to expect that the analytical framework we take in Section \ref{sec_convergence} can be  extended to a more general setting, as long as the stochastic gradient oracle satisfies  Condition \ref{condition_noise} and the objective function satisfies properties discussed in Lemma \ref{lemma_lqr_landscape}.

%% file: convergence.tex


\section{Convergence Analysis}\label{sec_convergence}

We start by noting that in view of Condition \ref{condition_noise}, to ensure that the stochastic gradient $G_t$ remains bounded throughout the optimization process, one needs to in turn control the optimality gap $\Delta(K_t)$. 
On the other hand, the optimality gap of the later policies would clearly depend on $G_t$.
In turn this suggests somewhat an induction approach to the convergence analysis. 
Nevertheless, it is clear that with our goal of taking only $\tilde{\cO}(1)$ interactions for the construction of $G_t$, 
it would be difficult to maintain such a high-probability certificate if we were to show that $\Delta(K_{t+1})$ remains bounded with high probability simply from the fact that $\Delta(K_t)$ is bounded with high probability and that we have an approximate descent at iteration $t$. 
In principle, $\tilde{\cO}(1)$ interactions for gradient evaluation typically ensure at best approximate descent in expectation, instead of with high probability.  

To proceed, our basic idea is to introduce the following auxiliary sequence $\cbr{Y_t}$, which in some sense generalizes the observation underlying Lemma 5.4 of \cite{li2025policy}.\footnote{
It is well known that the summation of $T$ independent standard Gaussian random variables is of order $\cO(\sqrt{T})$ with high probability.
Now suppose in round $t+1$ of summation, the random variable to be summed is Gaussian conditioned on the event that the partial sum of the first $t$ rounds is bounded by $C \sqrt{t}$ for some constant $C$, and $\infty$ otherwise. Lemma 5.4 in \cite{li2025policy} shows that the total sum is still of order $\cO(\sqrt{T})$ provided $C$ is large enough. 
}
In essence, since the behavior of $\Delta(K_t)$ depends on the quality of $\cbr{G_i}_{i < t}$, we will let $E_t$ denote the event of a bounded optimality gap for all iterations up to $t$, such that one can readily control $Y_t \coloneqq \Delta(K_t) \mathbbm{1}_{E_t}$.
Hence the remaining  goal is to show that $\Delta(K_t) = Y_t$ with high probability, i.e.,  the indicator $\mathbbm{1}_{E_t}$ does not really matter in terms of our probabilistic statement. 

To make our above observation precise, let us consider choosing $M, \Delta$ and $\cbr{\Delta_t}$ such that  
\begin{align}\label{basic_requirement_delta_seq}
M > 0, ~  \Delta(K_0) \leq \min \cbr{\Delta, \Delta_0},  ~
 0 \leq \Delta_t \leq \Delta,  ~ t \geq 0, 
 \end{align}
 and accordingly define 
\begin{align}
 E_t&=\cap_{i=0}^t\cbr{\Delta(K_i)\leq \Delta_i}
       \cap\cap_{i=0}^{t-1}\cbr{\norm{G_i}\leq M},  \label{def_good_event}\\
 \delta_t&=\nabla f(K_t)-G_t,\nonumber\\
 Y_0&=\Delta(K_0),~ Y_{t+1}=
 \Delta(K_{t+1}) \mathbbm{1}_{E_t\cap\cbr{\norm{G_t}\leq M}}. \label{def_auxiliary_gap}
\end{align}
We make the following immediate observation regarding the auxiliary sequence $\cbr{Y_t}$.

\begin{lemma}\label{lemma_aux_seq}
Suppose $0<\eta_t\leq 1/\mu$ and $\eta_tM\leq r_{\Delta}$ for some $M > 0$.
Then we have 
\begin{align}
 Y_{t+1}
 &\leq (1-\mu\eta_t)Y_t
       +\eta_t\inner{\nabla f(K_t)}
          {\delta_t\mathbbm{1}_{E_t\cap\cbr{\norm{G_t}\leq M}}}
       +\frac{L_{\Delta}}{2}\eta_t^2M^2.
       \label{ineq_single_step_gap_recursion}
\end{align}
\end{lemma}

\begin{proof}
On the event of $E_t\cap\cbr{\norm{G_t}\leq M}$, 
since $\Delta(K_t) \leq \Delta_t \leq \Delta$ and $\norm{K_{t+1} - K_t} = \eta_t \norm{G_t} \leq \eta_tM \leq r_{\Delta}$,
one can invoke \eqref{ineq_local_smoothness} in Lemma \ref{lemma_lqr_landscape} and obtain 
\begin{align*}
 f(K_{t+1})
 &\leq f(K_t)+\inner{\nabla f(K_t)}{K_{t+1}-K_t}
              +\frac{L_{\Delta}}{2}\norm{K_{t+1}-K_t}^2\\
 &=f(K_t)-\eta_t\inner{\nabla f(K_t)}{G_t}
              +\frac{L_{\Delta}}{2}\eta_t^2\norm{G_t}^2\\
 &=f(K_t)-\eta_t\norm{\nabla f(K_t)}^2
              +\eta_t\inner{\nabla f(K_t)}{\delta_t}
              +\frac{L_{\Delta}}{2}\eta_t^2\norm{G_t}^2.
\end{align*}
Applying the first inequality in \eqref{ineq_gradient_gap_bounds}, we have 
\begin{align*}
 \Delta(K_{t+1})
 \leq (1-\mu\eta_t)\Delta(K_t)
       +\eta_t\inner{\nabla f(K_t)}{\delta_t}
       +\frac{L_{\Delta}}{2}\eta_t^2\norm{G_t}^2.
\end{align*}
Consequently, for any $t \geq 1$, multiplying both sides by $\mathbbm{1}_{E_t\cap\cbr{\norm{G_t}\leq M}}$ and using the definition of $Y_t$ yields  
\begin{align}
 Y_{t+1}
 &\leq (1-\mu\eta_t)\Delta(K_t)
       \mathbbm{1}_{E_t\cap\cbr{\norm{G_t}\leq M}}  +\eta_t\inner{\nabla f(K_t)}
       {\delta_t \mathbbm{1}_{E_t\cap\cbr{\norm{G_t}\leq M}}}
       +\frac{L_{\Delta}}{2}\eta_t^2\norm{G_t}^2 \mathbbm{1}_{E_t\cap\cbr{\norm{G_t}\leq M}} \label{ineq_recusion_y_raw}  \\
  & \overset{(a)}{\leq} (1-\mu\eta_t)\Delta(K_t)
       \mathbbm{1}_{E_{t-1} \cap\cbr{\norm{G_{t-1}}\leq M}}
        +\eta_t\inner{\nabla f(K_t)}
       {\delta_t\mathbbm{1}_{E_t\cap\cbr{\norm{G_t}\leq M}}}
       +\frac{L_{\Delta}}{2}\eta_t^2 M^2  \nonumber \\
 & \leq (1-\mu\eta_t)Y_{t}
        +\eta_t\inner{\nabla f(K_t)}
       {\delta_t\mathbbm{1}_{E_t\cap\cbr{\norm{G_t}\leq M}}}
       +\frac{L_{\Delta}}{2}\eta_t^2 M^2,  \nonumber 
\end{align}
where $(a)$ follows from the fact that 
$
 E_t\cap\cbr{\norm{G_t}\leq M}
 \subseteq E_{t-1}\cap\cbr{\norm{G_{t-1}}\leq M}
$
and $\Delta(K_t) \geq 0$.
Hence \eqref{ineq_single_step_gap_recursion} holds for any $t \geq 1$. 
On the other hand, for $t = 0$, 
it is clear that 
\eqref{ineq_recusion_y_raw} still holds.
Combining this observation with the fact that 
$
(1-\mu\eta_0)\Delta(K_0)
       \mathbbm{1}_{E_0 \cap\cbr{\norm{G_0 }\leq M}}  \leq (1-\mu\eta_0)\Delta(K_0) = (1-\mu \eta_0) Y_0,
$
we conclude that \eqref{ineq_single_step_gap_recursion} holds at $t = 0$. 
\end{proof}

In view of Lemma \ref{lemma_aux_seq} and Condition \ref{condition_noise}, one can readily employ the basic Azuma inequality to control the auxiliary sequence $\cbr{Y_t}$ with high probability.
Whenever this is the case, we next  show that this in turn would lead to a control over $\cbr{\Delta(K_t)}$. 

\begin{lemma}\label{lemma_recovery_of_actual_gap}
Let $\Omega$ denote the sample space associated with $T$ iterations of SPG.
Suppose the random outcome $\omega$ satisfies\footnote{
For notational simplicity, we omit the explicit dependence of relevant random variables evaluated at $\omega$ within the statement here. 
} 
\begin{align}\label{restriction_omega}
Y_t \leq \Delta_t, ~ 0 \leq t < T; ~~~ \norm{G_t} \leq M ~ \text{if}~ \omega \in E_t, ~ 0 \leq t < T.
\end{align}
  Then for such $\omega$, we have
\begin{align*}
 \Delta(K_t) = Y_t, ~ \norm{G_{t-1}} \leq M, ~ 0\leq t\leq T,
\end{align*}
where we define $G_{-1} = 0$. 
\end{lemma}

\begin{proof}
Clearly the desired claim holds at $t=0$ given the definition of $Y_0 = \Delta(K_0)$. Suppose the claim holds at $t$.
That is, suppose for such $\omega$ we have $ \Delta(K_i) = Y_i$ for any $i \leq t$
and $\norm{G_i} \leq M$ for $i \leq t-1$. 
 Then for the considered $\omega$, we have 
\begin{align*}
\Delta(K_{t+1}) & \overset{(a)}{=} \Delta(K_{t+1}) \mathbbm{1}_{\cbr{Y_i \leq \Delta_i, i \leq t, \norm{G_i} \leq M, i \leq t-1}}  \\
& \overset{(b)}{=} \Delta(K_{t+1}) \mathbbm{1}_{\cbr{\Delta(K_i) \leq \Delta_i, i \leq t, \norm{G_i} \leq M, i \leq t-1 }} \\
& \overset{(c)}{=}  \Delta(K_{t+1}) \mathbbm{1}_{E_t}  \overset{(d)}{=}  \Delta(K_{t+1}) \mathbbm{1}_{E_t \cap \cbr{\norm{G_t} \leq M}} \overset{(e)}{=} Y_{t+1},
\end{align*}
where $(a)$ follows from the induction hypothesis that $\norm{G_i}\leq M$ for $i \leq t-1$,  and the requirement of $Y_i \leq \Delta_i$ for $i \leq t$ in \eqref{restriction_omega}, 
$(b)$ applies the induction hypothesis that $Y_i = \Delta(K_i)$ for $i \leq t$, 
$(c)$ applies the definition of $E_t$, 
and $(d)$ follows from the requirement of $\norm{G_t} \leq M$ provided $\omega \in E_t$, as indicated in \eqref{restriction_omega}.
Finally, $(e)$ follows directly from the definition of $Y_{t+1}$. 
\end{proof}

Lemma \ref{lemma_recovery_of_actual_gap} can be viewed as a generalization of the induction step presented in Lemma 5.4 of \cite{li2025policy}.
As will be shown in the subsequent Sections \ref{sec_sqrt_rate} and \ref{sec_sublinear_rate}, one can readily control the probability of the first requirement $\cbr{Y_t \leq \Delta_t, 0 \leq t < T}$ in \eqref{restriction_omega}.
We next present the following simple observation controlling the probability of the second part of the requirement in \eqref{restriction_omega}.

\begin{lemma}\label{lemma_union_prob_restriction}
Suppose Condition \ref{condition_noise} holds. We have
\begin{align} \label{prob_fail_cond_union}
\PP \cbr{
\norm{G_t} \leq M_{\Delta, \delta} ~ \text{if}~ \omega \in E_t, 0 \leq t <T
} 
\geq 
1 - T\delta . 
\end{align}
\end{lemma}

\begin{proof}
One can readily see that the set of random outcomes that do not satisfy the second part of the requirement in \eqref{restriction_omega} is given by 
$\cup_{0 \leq t <T}
E_t \cap \cbr{\norm{G_t} > M} 
$,
and hence 
\begin{align*}
\PP \cbr{
\norm{G_t} \leq M ~ \text{if}~ \omega \in E_t, 0 \leq t <T
} 
\geq 
1 - \tsum_{t<T} \PP \rbr{ \norm{G_t} > M | E_t  }.
\end{align*}
Since $\Delta_t \leq \Delta$, from the definition of $E_t$ it is clear that one could invoke \eqref{condition_bound_grad} in Condition \ref{condition_noise} to control the above probability with $M = M_{\Delta, \delta}$, which leads to 
the desired claim.
\end{proof}

Before we proceed to the convergence analysis of SPG, let us first recall the following slight modification of Azuma's inequality in view of Condition \ref{condition_noise}.

\begin{lemma}\label{lemma_biased_azuma}
Let $\cbr{\cF_t}$ be a filtration,  $Z_t$ be $\cF_{t+1}$-measurable, and $\theta_t$ be $\cF_t$-measurable. 
Suppose there exist $M > 0$ and $d_t > 0$ such that 
$\abs{Z_t} \leq M$ and $\abs{\theta_t} \leq d_t$. 
Define
$
 S_t=\tsum_{i < t } {\theta_i}{Z_i}.
$
Then for any $\delta\in(0,1)$, with probability at least $1-T\delta$,
we have
\begin{align*}
 S_t
 &\leq \tsum_{i<t} {\theta_i}{\EE[Z_i| \cF_i]}
       + 4 M \sqrt{\log(1/\delta)\tsum_{i<t}d_i^2},
\end{align*}
for $1\leq t \leq T$.
\end{lemma}

\subsection{SPG with $O(1/\sqrt{T})$ Rate}\label{sec_sqrt_rate}

We now proceed to first establish that SPG converges at the rate of $\cO(1/\sqrt{T})$ with high probability.
To this end, we begin by establishing the following basic recursion on the convergence of $\cbr{Y_t}$.  

\begin{lemma}\label{lemma_weighted_gap_recursion}
Suppose the stepsize $\cbr{\eta_t}$ satisfies 
\begin{align}\label{cond_stepsize_sqrt}
\eta_t \leq \frac{1}{\mu}, ~~  \eta_t M \leq r_{\Delta}.
\end{align}
Define $\Lambda_0 = 1$ and 
$
 \Lambda_{t+1}=\frac{\Lambda_t}{1-\mu\eta_t}.
$
Then we have
 \begin{align}
 Y_t  \leq \frac{1}{\Lambda_t} \rbr{ \Delta(K_0)+
       \tsum_{i<t}\Lambda_{i+1}\eta_i
       \inner{\nabla f(K_i)}{\delta_i\mathbbm{1}_{E_i\cap\cbr{\norm{G_i}\leq M}}}
      +\frac{L_{\Delta}M^2}{2}\tsum_{i<t}\Lambda_{i+1}\eta_i^2 } 
       \label{ineq_weighted_gap_recursion}
\end{align}
for $0 \leq t \leq T$. 
\end{lemma}

\begin{proof}
Multiplying both sides of \eqref{ineq_single_step_gap_recursion} in Lemma \ref{lemma_aux_seq} by $\Lambda_{t+1}$, and utilizing 
$(1-\mu \eta_t) \Lambda_{t+1} = \Lambda_t$, 
we obtain 
\begin{align*}
 \Lambda_{t+1}Y_{t+1}
 \leq  \Lambda_tY_t +  \Lambda_{t+1}\eta_t \inner{\nabla f(K_t)}
          {\delta_t \mathbbm{1}_{E_t \cap\cbr{\norm{G_t}\leq M}}} +\frac{L_{\Delta}M^2}{2}\Lambda_{t+1}\eta_t^2.
\end{align*}
The desired claim then follows from taking the telescopic sum of the above relation.
\end{proof}

Before we proceed directly to the convergence of $\cbr{K_t}$, we  establish below  the convergence of $\cbr{Y_t}$ such that $Y_t \leq \Delta_t' \leq \Delta$ for some properly chosen $\Delta_t' $ and $\Delta > 0$.
Indeed, suppose this is true. By taking $\Delta_t \equiv \Delta$ in \eqref{def_good_event}, then from Lemmas \ref{lemma_recovery_of_actual_gap} and \ref{lemma_union_prob_restriction}, it is clear that 
$\Delta(K_t) \leq \Delta_t'$ with high probability.
This also illustrates the basic roles of two sequences $\cbr{\Delta_t}$ and $\cbr{\Delta_t'}$ in our subsequent analysis. 
Namely, one can view $\Delta_t$ as the prescribed optimality gap that $K_t$ needs to attain for the optimality-gap dependent stochastic gradients to behave in a benign way, 
and $\Delta_t'$ is the actual rate of convergence if the former holds true.

\begin{lemma}\label{lemma_azuma_fixed_gap_threshold}
Suppose the stepsize $\cbr{\eta_t}$ satisfies \eqref{cond_stepsize_sqrt} with $M = M_{\Delta, \delta}$ therein, 
and Condition \ref{condition_noise} holds. 
Let ${\Lambda_t}$ be defined as in Lemma \ref{lemma_weighted_gap_recursion}. 
Define $\Delta_0' = \Delta(K_0)$, and 
\begin{align}
\Delta_t' =   \frac{\Delta(K_0)}{\Lambda_t}
   +\frac{G_\Delta\rbr{b+\sqrt{V_\Delta\delta}}}{\mu}
 +\frac{4G_\Delta(M_{\Delta,\delta}+G_\Delta)}{\Lambda_t}
      \sqrt{\log(1/\delta)\tsum_{i<t}\Lambda_{i+1}^2\eta_i^2}
   +\frac{L_{\Delta}M_{\Delta,\delta}^2}{2\Lambda_t}\tsum_{i<t}\Lambda_{i+1}\eta_i^2 , ~  1 \leq t \leq T. 
   \label{def_azuma_gap_envelope_sqrt}
\end{align}
Suppose $\Delta_t' \leq \Delta$ for every $1\leq t\leq T$. Then with probability
at least $1-2T\delta$, we have
\begin{align*}
 \Delta(K_t)\leq \Delta_t',~~ 0\leq t\leq T.
\end{align*}
\end{lemma}

\begin{proof}
Consider choosing $\Delta_t = \Delta$ in \eqref{basic_requirement_delta_seq} and $M = M_{\Delta, \delta}$ in  \eqref{def_good_event}. 
Suppose $\Delta (K_i) \leq \Delta_i = \Delta$. Then in view of \eqref{bd_gd_via_opt_gap}, we have 
 $\norm{\nabla f(K_i)} \leq G_\Delta$, and consequently 
\begin{align}\label{ineq_sqrt_azuma_abs_bd}
 \abs{\inner{\nabla f(K_i)}
       {\delta_i\mathbbm{1}_{E_i\cap\cbr{\norm{G_i}\leq M_{\Delta,\delta}}}}}
 &\leq G_\Delta(M_{\Delta,\delta}+G_\Delta).
\end{align}
In addition, we have 
\begin{align}
 &\EE\sbr{\inner{\nabla f(K_i)}
       {\delta_i\mathbbm{1}_{E_i\cap\cbr{\norm{G_i}\leq M_{\Delta,\delta}}}}| \cF_i}
       \nonumber\\
 = & \EE\sbr{\inner{\nabla f(K_i)}{\delta_i\mathbbm{1}_{E_i}}
       | \cF_i}  -\EE\sbr{\inner{\nabla f(K_i)}
       {\delta_i\mathbbm{1}_{E_i\cap\cbr{\norm{G_i}>M_{\Delta,\delta}}}}| \cF_i}
       \nonumber\\
 \overset{(a)}{\leq} &  bG_\Delta+G_\Delta
       \EE[\norm{G_i}\mathbbm{1}_{E_i\cap\cbr{\norm{G_i}>M_{\Delta,\delta}}}| \cF_i]
       \nonumber\\
 \overset{(b)}{\leq} &  bG_\Delta+G_\Delta
       \sqrt{\EE[\norm{G_i}^2\mathbbm{1}_{E_i}| \cF_i] 
             \PP(E_i\cap\cbr{\norm{G_i}>M_{\Delta,\delta}}| \cF_i)}
       \nonumber\\
 \overset{(c)}{\leq} &  \rbr{b+\sqrt{V_\Delta\delta}}G_\Delta,
 \label{ineq_truncated_inner_product_bias}
\end{align}
where $(a)$ follows from \eqref{condition_bound_bias} in Condition \ref{condition_noise} and the Cauchy-Schwarz inequality, 
and $(b)$ follows again from the Cauchy-Schwarz inequality, 
and $(c)$ follows from \eqref{condition_bound_moment} and \eqref{condition_bound_grad} in Condition \ref{condition_noise}.
Consequently, by combining  Lemma \ref{lemma_biased_azuma} with observations \eqref{ineq_sqrt_azuma_abs_bd} and \eqref{ineq_truncated_inner_product_bias} above,
 we obtain 
that with probability at least $1-T\delta$,
\begin{align*}
 &\tsum_{i<t}\Lambda_{i+1}\eta_i
       \inner{\nabla f(K_i)}
          {\delta_i\mathbbm{1}_{E_i\cap\cbr{\norm{G_i}\leq M_{\Delta,\delta}}}}\\
\leq &  \rbr{b+\sqrt{V_\Delta\delta}}G_\Delta\tsum_{i<t}\Lambda_{i+1}\eta_i
       +4G_\Delta(M_{\Delta,\delta}+G_\Delta)\sqrt{\log(1/\delta)\tsum_{i<t}\Lambda_{i+1}^2\eta_i^2}  \\ 
        \leq  &
   \frac{\Lambda_t G_\Delta\rbr{b+\sqrt{V_\Delta\delta}}}{\mu}
       +4G_\Delta(M_{\Delta,\delta}+G_\Delta)\sqrt{\log(1/\delta)\tsum_{i<t}\Lambda_{i+1}^2\eta_i^2}  
\end{align*}
for any $1\leq t\leq T$,
where the last inequality follows from $ \tsum_{i<t}\Lambda_{i+1}\eta_i
   =\frac{1}{\mu}\tsum_{i<t}(\Lambda_{i+1}-\Lambda_i)
 \leq \frac{\Lambda_t}{\mu}.
$
Combining the above observation with the definition of $\Delta_t'$ in \eqref{def_azuma_gap_envelope_sqrt}, and \eqref{ineq_weighted_gap_recursion} in Lemma \ref{lemma_weighted_gap_recursion} implies $Y_t \leq \Delta_t'$ for all $0 \leq t \leq T$ with probability $1-T\delta$.
Now since $\Delta_t' \leq \Delta$,
we have $Y_t \leq \Delta$.
 In view of Lemma \ref{lemma_recovery_of_actual_gap},
and the fact that  $\norm{G_i} \leq M_{\Delta, \delta}$ with probability at least $1- \delta$ conditioned on $E_i$, 
we obtain 
$
\Delta (K_t) = Y_t \leq \Delta_t' 
$
with probability at least $1-2T\delta$ for $0 \leq t \leq T$.
%
\end{proof}

We are now ready to specify the concrete choice of stepsize for SPG and establish its $\cO(1/\sqrt{T})$ convergence with high probability.

\begin{theorem}\label{thrm_sqrt_rate}
Let $\Delta=4\Delta(K_0)$.
 Suppose Condition \ref{condition_noise} holds and  
$G_\Delta\rbr{b+\sqrt{V_\Delta\delta}}/\mu\leq \Delta(K_0)$. 
Denote 
\begin{align*}
 N =\ceil{\max\cbr{4,\frac{2M_{\Delta,\delta}}{\mu r_{\Delta}},
       \frac{1024G_\Delta^2(M_{\Delta,\delta}+G_\Delta)^2\log(1/\delta)}{\mu^2\Delta(K_0)^2},
       \frac{8L_{\Delta}M_{\Delta,\delta}^2}{\mu^2\Delta(K_0)}}}, ~~~ C =3\Delta(K_0)\sqrt N.
\end{align*}
Set 
$
 \eta_t=\frac{2}{\mu(N+t)}. 
$
Then  with probability at least $1-2T\delta$, we have 
\begin{align}
 \Delta(K_t)\leq \frac{G_\Delta\rbr{b+\sqrt{V_\Delta\delta}}}{\mu}+\frac{C}{\sqrt{N+t}},
 \label{ineq_policy_gap_sqrt_rate}
\end{align}
 for $0\leq t\leq T$.
\end{theorem}

\begin{proof}
Clearly, the choice of $\eta_t$ satisfies \eqref{cond_stepsize_sqrt}. Consequently, 
in view of Lemma \ref{lemma_azuma_fixed_gap_threshold}, it suffices to show that $\Delta_t' \leq \Delta$ for $1\leq t \leq T$,
where $\Delta_t' $ is defined as in \eqref{def_azuma_gap_envelope_sqrt}.
One can readily verify that
$
 \Lambda_t=\frac{(N+t-2)(N+t-1)}{(N-2)(N-1)}
$~and
\begin{align}
 \frac{4G_\Delta(M_{\Delta,\delta}+G_\Delta)}{\Lambda_t}
       \sqrt{\log(1/\delta)\tsum_{i<t}\Lambda_{i+1}^2\eta_i^2}
 &\leq \frac{32G_\Delta(M_{\Delta,\delta}+G_\Delta)\sqrt{\log(1/\delta)}}
              {\mu\sqrt{N+t}}, \nonumber \\
 \frac{L_{\Delta}M_{\Delta,\delta}^2}{2\Lambda_t}\tsum_{i<t}\Lambda_{i+1}\eta_i^2
 &\leq \frac{8L_{\Delta}M_{\Delta,\delta}^2}{\mu^2(N+t)}. \label{sum_second_moment_grad}
\end{align}
Combining the above observations with \eqref{def_azuma_gap_envelope_sqrt}, we obtain
\begin{align}
 \Delta_t'
 &\leq \frac{\Delta(K_0)}{\Lambda_t}+\frac{G_\Delta\rbr{b+\sqrt{V_\Delta\delta}}}{\mu} +\frac{32G_\Delta(M_{\Delta,\delta}+G_\Delta)\sqrt{\log(1/\delta)}}
                    {\mu\sqrt{N+t}}
       +\frac{8L_{\Delta}M_{\Delta,\delta}^2}{\mu^2(N+t)} \label{ineq_sqrt_opt_gap_raw}\\
 &\leq \Delta(K_0)+\frac{G_\Delta\rbr{b+\sqrt{V_\Delta\delta}}}{\mu} +\frac{32G_\Delta(M_{\Delta,\delta}+G_\Delta)\sqrt{\log(1/\delta)}}
                    {\mu\sqrt N}
       +\frac{8L_{\Delta}M_{\Delta,\delta}^2}{\mu^2N} \nonumber \\
 &\leq 4\Delta(K_0) = \Delta, \nonumber 
\end{align}
for $1\leq t\leq T$,
where the last inequality follows from the choice of $N$. 
Hence applying Lemma \ref{lemma_azuma_fixed_gap_threshold} yields that $\Delta(K_t) \leq \Delta_t'$ with probability at least $1- 2 T\delta$. 
The desired claim then follows from  the definition of $\Lambda_t$ and $C$, together with \eqref{ineq_sqrt_opt_gap_raw}.
\end{proof}

In view of Theorem \ref{thrm_sqrt_rate}, we can immediately obtain the following $\cO(1/\epsilon^2)$ iteration complexity of SPG for finding an $\epsilon$-optimal policy.

\begin{corollary}\label{corr_sample_complexity_sqrt}
Suppose Condition \ref{condition_noise} holds. Choose the stepsize $\cbr{\eta_t}$ as in Theorem \ref{thrm_sqrt_rate}. 
For any $0<\epsilon\leq \Delta(K_0)$ and $0<\delta<1$ such that
\begin{align}
 \frac{G_\Delta}{\mu}\rbr{b+\sqrt{V_\Delta\delta}}
 \leq \frac{\epsilon}{2}, \label{condition_bias_accuracy_sqrt}
\end{align}
 with probability at least $1-2T\delta$, we have $f(K_T) - f^* \leq \epsilon$ 
if 
$T\geq {1 + {\frac{4C^2}{\epsilon^2}}}$,
and   $K_t$ is stable for any $0 \leq t \leq T$. 
\end{corollary}

It should be noted that condition \eqref{condition_bias_accuracy_sqrt} can be readily satisfied if the bias $b$ and $\delta$ associated with the stochastic gradient $G_t$ are sufficiently small. 
In Section \ref{sec_oracle}, we will discuss a detailed construction of $G_t$ that certifies both requirements. In  particular, such a construction requires only two trajectories each with length that depends logarithmically on $\max \cbr{1/b, 1/\delta}$.

\subsection{SPG with $O(1/{T})$ Rate}\label{sec_sublinear_rate}

We now discuss a refinement to our approach in Section \ref{sec_sqrt_rate}, which leads to an improved $\cO(1/T)$ convergence rate for the SPG method. 
The basic observation we utilize is the fact that the upper bound  in \eqref{ineq_sqrt_azuma_abs_bd} used in Azuma's inequality indeed 
depends on the optimality gap $\Delta(K_t)$, which in view of \eqref{bd_gd_via_opt_gap}, 
also diminishes provided that $\cbr{K_t}$ converges to the optimal policy.

\begin{lemma}\label{lemma_azuma_varying_gap_thresholds}
Suppose the stepsize $\cbr{\eta_t}$ satisfies \eqref{cond_stepsize_sqrt}
with $M=M_{\Delta,\delta}$, and Condition \ref{condition_noise} holds.
Let $\Lambda_t$ be defined as in Lemma \ref{lemma_weighted_gap_recursion}.
 Define $\Delta_0'=\Delta(K_0)$, and
\begin{align}
\Delta_t' =   \frac{\Delta(K_0)}{\Lambda_t}
   +\frac{G_\Delta\rbr{b+\sqrt{V_\Delta\delta}}}{\mu}
 +\frac{8(M_{\Delta,\delta}+G_\Delta)}{\Lambda_t}
      \sqrt{L_{\Delta}\log(1/\delta)\tsum_{i<t}\Lambda_{i+1}^2\eta_i^2\Delta_i}
   +\frac{L_{\Delta}M_{\Delta,\delta}^2}{2\Lambda_t}\tsum_{i<t}\Lambda_{i+1}\eta_i^2 , 
   \label{def_azuma_gap_envelope_fast_rate}
\end{align}
for $ 1 \leq t \leq T$.
Suppose $\Delta_t' \leq \Delta_t$ for every $1\leq t\leq T$. Then with probability
at least $1-2T\delta$, we have
\begin{align*}
 \Delta(K_t)\leq \Delta_t',~0\leq t\leq T.
\end{align*}
\end{lemma}

\begin{proof}
Note that from \eqref{bd_gd_via_opt_gap} we have 
\begin{align*}
 \abs{\inner{\nabla f(K_i)}
       {\delta_i\mathbbm{1}_{E_i\cap\cbr{\norm{G_i}\leq M_{\Delta,\delta}}}}}
 \leq (M_{\Delta,\delta}+G_\Delta)\norm{\nabla f(K_i)}\mathbbm{1}_{E_i} 
 \leq 2(M_{\Delta,\delta}+G_\Delta)\sqrt{L_{\Delta}\Delta_i}.
\end{align*}
Since $\Delta_t \leq \Delta$ as required in \eqref{basic_requirement_delta_seq}, in view of the definition of $E_t$ in \eqref{def_good_event}, the argument following \eqref{ineq_sqrt_azuma_abs_bd} in  Lemma~\ref{lemma_azuma_fixed_gap_threshold} still applies, from which the desired claim follows. 
\end{proof}

With Lemma \ref{lemma_azuma_varying_gap_thresholds} in place, one can then improve the convergence of SPG to $\cO(1/T)$ by following similar lines as in Theorem \ref{thrm_sqrt_rate}.

\begin{theorem}\label{thrm_fast_rate}
Let $\Delta=6\Delta(K_0)$.
Suppose Condition \ref{condition_noise} holds and
$G_\Delta\rbr{b+\sqrt{V_\Delta\delta}}/\mu\leq \Delta(K_0)$.
Denote
\begin{align}\label{def_fast_rate_agg_coeff}
 D=\frac{8L_{\Delta}}{\mu^2}
       \sbr{M_{\Delta,\delta}^2+768(M_{\Delta,\delta}+G_\Delta)^2\log(1/\delta)}, ~
 N=\ceil{\max\cbr{4,\frac{2M_{\Delta,\delta}}{\mu r_{\Delta}},
                      \frac{D}{\Delta(K_0)}}}, ~ C=2(N\Delta(K_0)+D).
\end{align}
Set
$
 \eta_t=\frac{2}{\mu(N+t)}. 
$
Then with probability at least $1-2T\delta$, we have
\begin{align}
 \Delta(K_t)\leq \frac{2G_\Delta\rbr{b+\sqrt{V_\Delta\delta}}}{\mu}  +\frac{C}{N+t}, 
 \label{ineq_policy_gap_fast_rate}
\end{align}
for $0\leq t\leq T$.
\end{theorem}

\begin{proof}
Let $\beta=G_\Delta\rbr{b+\sqrt{V_\Delta\delta}}/\mu$.
Clearly we have $\beta \leq \Delta(K_0)$.  
Consider choosing $\cbr{\Delta_t}$ as
\begin{align*}
 \Delta_t=2\beta+\frac{C}{N+t},~ 0\leq t\leq T.
\end{align*}
We first proceed to show that the defined $\cbr{\Delta_t}$ and $\Delta$ satisfy \eqref{basic_requirement_delta_seq}. 
It can be readily verified that we have  $\Delta(K_0)\leq \Delta_0$, $\Delta(K_0) \leq \Delta$, and that 
\begin{align*}
 \Delta_i
 \leq 2\beta+\frac{C}{N+i}
 \leq 2\beta+2\Delta(K_0)+\frac{2D}{N}  \leq 6\Delta(K_0) = \Delta. 
\end{align*}
Clearly, the choice of stepsize $\cbr{\eta_t}$ satisfies the condition in Lemma \ref{lemma_azuma_varying_gap_thresholds}, and hence it remains to verify that $\Delta_t' \leq \Delta_t$. 
As in the proof of Theorem \ref{thrm_sqrt_rate}, 
one can readily verify that
$
 \Lambda_t=\frac{(N+t-2)(N+t-1)}{(N-2)(N-1)}
$, and
\begin{align*}
 \tsum_{i<t}\rbr{\frac{\Lambda_{i+1}\eta_i}{\Lambda_t}}^2\Delta_i
 \leq \frac{128\beta}{3\mu^2(N+t)}
       +\frac{32C}{\mu^2(N+t)^2}.
\end{align*}
Consequently, we obtain 
\begin{align*}
 &\frac{8(M_{\Delta,\delta}+G_\Delta)}{\Lambda_t}
       \sqrt{L_{\Delta}\log(1/\delta)\tsum_{i<t}\Lambda_{i+1}^2\eta_i^2\Delta_i}\\
\overset{(a)}{\leq} &  \frac{64(M_{\Delta,\delta}+G_\Delta)}{\mu}
       \sqrt{\frac{L_{\Delta}\beta\log(1/\delta)}{N+t}}
       +\frac{64(M_{\Delta,\delta}+G_\Delta)
                     \sqrt{L_{\Delta}C\log(1/\delta)}}{\mu(N+t)}\\
\overset{(b)}{\leq} & \frac{\beta}{2}+\frac{C}{4(N+t)}
       +\frac{6144L_{\Delta}(M_{\Delta,\delta}+G_\Delta)^2\log(1/\delta)}{\mu^2(N+t)}, 
\end{align*}
where $(a)$ uses the inequality $\sqrt{a + b} \leq \sqrt{a} + \sqrt{b}$, and $(b)$ applies $ab \leq (a^2 + b^2 )/2$. 
Note that \eqref{sum_second_moment_grad} still holds.
Hence by combining the above observation with \eqref{sum_second_moment_grad}, 
together with the definition of $\Delta_t'$ in \eqref{def_azuma_gap_envelope_fast_rate}, 
we obtain 
\begin{align}
 \Delta_t' & \leq   \frac{N\Delta(K_0)}{N+t}+ \frac{3\beta}{2} +\frac{8L_{\Delta}M_{\Delta,\delta}^2}{\mu^2(N+t)} + \frac{C}{4(N+t)}
       +\frac{6144L_{\Delta}(M_{\Delta,\delta}+G_\Delta)^2\log(1/\delta)}{\mu^2(N+t)}  \label{fast_rate_opt_gap_raw} \\
& \leq 
\frac{3\beta}{2}+\frac{N\Delta(K_0)+C/4+D}{N+t}
 =\frac{3\beta}{2}+\frac{3C}{4(N+t)}
 \leq 2\beta+\frac{C}{N+t}=\Delta_t \nonumber
\end{align}
for $1 \leq t \leq T$. 
Hence applying Lemma \ref{lemma_azuma_varying_gap_thresholds} yields that $\Delta(K_t) \leq \Delta_t'$ with probability at least $1- 2 T\delta$. 
The desired claim then follows from  
\eqref{fast_rate_opt_gap_raw} with the choice of $C$ and $\beta$. 
\end{proof}

In view of Theorem \ref{thrm_fast_rate}, we can immediately obtain the following $\cO(1/\epsilon)$ iteration complexity of SPG for finding an $\epsilon$-optimal policy.

\begin{corollary}\label{corr_sample_complexity_fast_rate}
Suppose Condition \ref{condition_noise} holds. Choose the stepsize $\cbr{\eta_t}$ as in Theorem \ref{thrm_fast_rate}. 
For any $0<\epsilon\leq \Delta(K_0)$ and $0<\delta<1$ such that
\begin{align}
 \frac{2G_\Delta}{\mu}\rbr{b+\sqrt{V_\Delta\delta}}
 \leq \frac{\epsilon}{2}, \label{condition_bias_accuracy_fast_rate}
\end{align}
 with probability at least $1-2T\delta$, we have $f(K_T) - f^* \leq \epsilon$
if
$T\geq \max\cbr{1,\frac{2C}{\epsilon}}$,
and $K_t$ is stable for any $0 \leq t \leq T$.
\end{corollary}

%% file: oracle.tex


\section{Construction of Stochastic Oracle}\label{sec_oracle}

Let us recall the following construction of the two-point estimator first developed for convex optimization  \cite{agarwal2010optimal, duchi2015optimal, shamir2017optimal}  and later studied in the context of LQR \cite{malik2020derivative}.
Specifically, for a given $\alpha > 0$, let us randomly sample $x_0 \sim \cD$, and $U$ uniformly distributed over the unit sphere in $\RR^{m \times n}$, 
and construct 
\begin{align} \label{def_two_point_gradient_estimator}
 G^\alpha(K) = \frac{d}{2\alpha} \rbr{F(K+\alpha U;x_0) - F(K-\alpha U;x_0)}U,
\end{align}
where $d=mn$,
and $F(K; x) = \tsum_{i \geq 0} x_i^\top (Q + K^\top R K) x_i$ where $x_i = (A-BK)^i x$.
As will be seen in the ensuing Proposition \ref{prop_two_point_ideal}, the two-point estimator generalizes the classical one-point estimator for zeroth-order convex optimization \cite{flaxman2004online,nemirovskij1983problem} with the benefit of reduced variance  that can be made independent of $\alpha$ \cite{agarwal2010optimal, duchi2015optimal, shamir2017optimal}.

\begin{proposition}\label{prop_two_point_ideal}
For any $\Delta > 0$, let $L_\Delta, r_\Delta$ be defined in Lemma \ref{lemma_lqr_landscape}. 
There exists $\ell_\Delta \geq 0$ such that for any $0 < \alpha \leq r_\Delta$, 
the estimator $G^\alpha(K)$ defined in \eqref{def_two_point_gradient_estimator} satisfies Condition \ref{condition_noise} 
with 
\begin{align}\label{def_two_point_oracle_parameters}
 b=L_{\Delta}\alpha,~V_\Delta=d^2\ell_\Delta^2,~M_{\Delta,\delta}=d\ell_\Delta, 
\end{align}
for any $0 < \delta < 1$. 
\end{proposition}

The proof of Proposition \ref{prop_two_point_ideal} can be found in \cite[Lemma 14 and Corollary 10]{malik2020derivative}.
It should be noted that the norm of the stochastic gradient $M_{\Delta, \delta}$ is independent of the confidence level $\delta$ for $G^\alpha(K)$ defined above,
and \eqref{condition_bound_grad} holds almost surely in this case. 
Of course, $G^\alpha (K)$ should be viewed as a conceptual estimator as $F(K; x)$ requires collecting a trajectory of infinite length. 
In this case, we show that $F(K; x)$ can be approximated well via a trajectory of length $h$, uniformly for all $K$ with bounded optimality gap.

\begin{lemma}\label{lemma_lqr_truncation}
For any stable policy $K$ and any $x \in \RR^n$, 
define  
$
 F_h (K; x) = \tsum_{i < h} x_i^\top (Q + K^\top R K) x_i, 
$
where $x_i = (A - BK)^i x$.
For any $\Delta > 0$, let  $f_\Delta = f^* + \Delta$, and suppose $\Delta(K) \leq \Delta$. Then 
\begin{align}\label{ineq_two_point_return_tail}
0 \leq F(K; x) - F_h(K; x) \leq 
\rbr{
1 - \frac{\lambda_{\min} (Q) }{ f_\Delta }
}^h  f_\Delta \norm{x}^2.
\end{align}
\end{lemma}

\begin{proof}
Since $K$ is stable, there exists  $P_K \succ 0$ being the unique solution satisfying 
\begin{align}\label{lqr_dp_eval}
 P_K=Q+K^\top RK+(A-BK)^\top P_K(A-BK).
\end{align}
In addition, we have 
$F(K; x) = x^\top P_K x$, and $f(K) = \EE_{x_0} F(K; x_0) = \mathrm{tr}(P_K)$,
where the last equality follows from \eqref{eq_dynamics}. 
From this we obtain 
\begin{align*}
F(K; x) \leq \norm{P_K} \norm{x x^\top} \leq \mathrm{tr} (P_K) \norm{x}^2 \leq f(K) \norm{x}^2.
\end{align*}
Combining the above observation with \eqref{lqr_dp_eval}, we have 
\begin{align*}
F(K; (A-BK) x) = F(K; x) - x^\top (Q + K^\top R K) x 
\leq F(K; x) - \lambda_{\min} (Q) \norm{x}^2 
\leq \rbr{
1 - \frac{\lambda_{\min} (Q) }{ f(K) }
} F(K; x).
\end{align*}
Recursive application of the above inequality then yields
\begin{align*}
F(K; x) - F_h(K;x) = F(K; (A-BK)^h x) \leq \rbr{
1 - \frac{\lambda_{\min} (Q) }{ f(K) }
}^h F(K; x)
\leq 
\rbr{
1 - \frac{\lambda_{\min} (Q) }{ f(K) }
}^h  f(K) \norm{x}^2 
\end{align*}
from which the desired claim follows. 
\end{proof}

With Lemma \ref{lemma_lqr_truncation} in place, let us consider  
\begin{align}\label{def_finite_two_point_gradient_estimator}
 G_h^\alpha(K)=\frac{d}{2\alpha}\rbr{F_h(K+\alpha U;x_0)-F_h(K-\alpha U;x_0)}U.
\end{align}
We proceed to show that $ G_h^\alpha(K)$ defined above can again certify Condition \ref{condition_noise} with a proper choice of truncation length $h$.

\begin{proposition}\label{prop_finite_two_point_gradient_oracle}
For any $\Delta>0$, let $L_\Delta,r_\Delta,\ell_\Delta$ be as in Proposition \ref{prop_two_point_ideal}.
For any $0<\alpha\leq r_\Delta$ and integer $h\geq1$, define
$
 \varepsilon_h=\frac{d f_{2\Delta}D_X^2}{2\alpha}\rbr{1-\frac{\lambda_{\min}(Q)}{f_{2\Delta}}}^h .
$
Then the estimator $G_h^\alpha(K)$ in \eqref{def_finite_two_point_gradient_estimator} satisfies Condition \ref{condition_noise} with
\begin{align}\label{def_finite_two_point_oracle_parameters}
 b=L_\Delta\alpha+\varepsilon_h,~V_\Delta=(d\ell_\Delta+\varepsilon_h)^2,~M_{\Delta,\delta}=d\ell_\Delta+\varepsilon_h,
\end{align}
for any $0<\delta<1$.
\end{proposition}

\begin{proof}
In view of Lemma \ref{lemma_lqr_landscape} and \eqref{bd_gd_via_opt_gap}, for any $K$ with $\Delta(K) \leq \Delta$,  we have 
\begin{align*}
f(K + \alpha U) \leq f(K) + \alpha G_\Delta + \frac{L_\Delta \alpha^2}{2} 
\leq f^* + 2\Delta
\end{align*}
for any $\alpha \leq r_\Delta$.\footnote{One can without loss of generality assume that $r_\Delta$ is small enough
such that $r_\Delta G_\Delta + \frac{L_\Delta r_\Delta^2}{2} \leq \Delta$. 
}
Hence Lemma \ref{lemma_lqr_truncation} applies to both $K+ \alpha U$ and $K-\alpha U$, 
from which the desired claim follows after combining with Proposition \ref{prop_two_point_ideal}.
\end{proof}

We are now ready to determine the number of transitions needed by SPG for finding an $\epsilon$-optimal policy, when the stochastic gradient estimator is given by \eqref{def_finite_two_point_gradient_estimator}.

\begin{corollary}\label{corr_two_point_transition_complexity}
Set $\Delta=6\Delta(K_0)$.
For any $\epsilon \leq \Delta(K_0)$ and $\delta \in (0,1)$ such that
$
 \delta\leq\rbr{\frac{\mu\epsilon}{16G_\Delta d\ell_\Delta}}^2,
$
let the stochastic gradient estimator be given by \eqref{def_finite_two_point_gradient_estimator}, with parameters 
\begin{align}
 \alpha=\min\cbr{r_\Delta,\frac{\mu\epsilon}{16G_\Delta L_\Delta}}, ~~ 
 h =\max\cbr{1,\ceil{\frac{f_{2\Delta}}{\lambda_{\min}(Q)}\log\max\cbr{\frac{f_{2\Delta}D_X^2}{2\alpha\ell_\Delta},\frac{8d f_{2\Delta}D_X^2G_\Delta}{\alpha\mu\epsilon}}}}.
 \label{def_two_point_fast_rate_parameters}
\end{align}
Then Condition \ref{condition_noise} is satisfied with 
\begin{align*}
b=L_\Delta\alpha+\varepsilon_h \leq  \frac{\mu\epsilon}{8G_\Delta}, ~ V_\Delta=4d^2\ell_\Delta^2, ~ M_{\Delta,\delta}=2d\ell_\Delta.
\end{align*}
In addition, run SPG for a total of $T$ iterations with stepsize $\cbr{\eta_t}$ chosen as in Theorem \ref{thrm_fast_rate}.
Then with probability at least $1-2T\delta$, we have $f(K_T) - f^* \leq \epsilon$
if
$T\geq \max\cbr{1,\frac{2C}{\epsilon}}$,
and $K_t$ is stable for any $0 \leq t \leq T$.
Here $C$ is defined as in \eqref{def_fast_rate_agg_coeff}.

\end{corollary}
\begin{proof}
The first part of the claim follows directly from Proposition \ref{prop_finite_two_point_gradient_oracle} together with the fact that $L_\Delta\alpha\leq\mu\epsilon/(16G_\Delta)$, $\varepsilon_h\leq\mu\epsilon/(16G_\Delta)$, and $\varepsilon_h \leq d \ell_\Delta$ given the choice of $\alpha$ and $h$.
The second part of the claim then follows  from Corollary \ref{corr_sample_complexity_fast_rate}  by noting that  \eqref{condition_bias_accuracy_fast_rate} is satisfied
with the choice of $ \delta\leq\rbr{\frac{\mu\epsilon}{16G_\Delta d\ell_\Delta}}^2$.
\end{proof}

In view of Corollary \ref{corr_two_point_transition_complexity}, 
to obtain an $\epsilon$-optimal policy with probability at least $1-\delta$, 
it suffices to run SPG for 
$
\cO({C}{\epsilon}^{-1}) = \cO ({\log(1/\delta)}{\epsilon}^{-1})
$
iterations.
Combining this with the fact that $h = \cO(\log (\Delta_0/\epsilon))$, 
 the total number of environment interactions collected by SPG can be bounded by ${\cO}({\tt Polylog}(\epsilon^{-1}, \delta^{-1}) \epsilon^{-1})$.

%% file: conclusion.tex


\section{Concluding Remarks}

We now briefly discuss a few more questions of potential interest. 
First, it is well known that for \eqref{lqr_obj},  with model-based methods a constant number of environment interactions is enough for identifying the exact system parameters. 
Hence it would be interesting to study whether the reported $\cO(1/\epsilon)$ interaction complexity can be further improved. 
In addition,  it is also possible to extend the analysis to average-cost and discounted-cost settings more explicitly by considering detailed constructions of the stochastic gradient oracles satisfying Condition \ref{condition_noise}.